\documentclass[12pt]{article}
\usepackage{amsmath, amsthm}\usepackage{enumerate}\usepackage{amssymb}
\usepackage{bbold}
\usepackage{bbm}\usepackage{dsfont}\usepackage{color}\usepackage{xcolor}
\usepackage[explicit]{titlesec}
\newtheorem{theorem}{Theorem}[section]
\newtheorem{proposition}[theorem]{Proposition}

\newtheorem{lemma}[theorem]{Lemma}
\newtheorem{example}[theorem]{Example}
\newtheorem{corollary}[theorem]{Corollary}

\newtheorem{assertion}[theorem]{Assertion}

\newtheorem{definition}[theorem]{Definition}
\newcommand{\sol}{{\text{\rm sol}}}

\makeatletter
\renewcommand{\subsection}{\@startsection{subsection}{1}
{0pt}{3.25ex plus 1ex minus.2ex}{-1em}{\normalfont\normalsize\bf}}\makeatother
\begin{document}

\title{{\bf On limited and almost limited operators between Banach lattices}}
\maketitle
\author{\centering{{Safak Alpay$^{1}$, Svetlana Gorokhova $^{2}$\\ 
\small $1$ Middle East Technical University, 06800 Ankara, Turkey\\ 
\small $2$ Southern Mathematical Institute VSC RAS, 362025 Vladikavkaz, Russia}
\abstract{We study (almost) limited operators in Banach lattices and their relations to
L-weakly compact, semi-compact, and Dunford-Pettis operators. 
Several further related topics are investigated.}

\vspace{5mm}

\vspace{3mm}
}}
 
\medskip

\section{Introduction and Preliminaries}

\noindent
Limited operators and their relations to L(M)weakly compact, semi-compact, Dunford-Pettis, and other classes of 
operators attract permanent attention of many researchers 
(see, for example \cite{AEG_duality,AG_aoLwc,BKAB2024,ArdaChen23,EG23aff,OM22,AEG22} and references therein). 
The aim of this work is to find conditions to ensure inclusion relation among major classes of
operators that are in use. In trying to do so we also touch upon properties of Banach lattices and 
note a method of producing new classes of operators ii them. Collectively qualified families of operators 
were studied recently in \cite{Em24coll_order,Em24coll,Em24coLevi,AEG_collect1}. In the end of the paper we 
investigate the collective versions of those results.

Throughout the paper, all vector spaces are real, operators are linear; letters $X$, $Y$, and $Z$ 
denote Banach spaces; $E$, $F$, and $G$ denote Banach lattices. A subset $C$ of $X$ is called
{\em bounded} if $C$ is norm bounded. We denote by $B_X$ the closed unit ball of $X$,
by $\text{\rm L}(X,Y)$ ($\text{\rm K}(X,Y)$, $\text{\rm W}(X,Y)$) 
the space of bounded (compact, weakly compact) operators 
from $X$ to $Y$, by $\text{\rm Id}_X$ the identity operator on $X$, 
by $E_+$ the positive cone of $E$, by $\text{sol}(A) = \bigcup_{a\in A}[-|a|,|a|]$ the solid hull 
of $A\subseteq E$, by $x_\alpha\downarrow 0$ a decreasing net $(x_\alpha)$ with $\inf_\alpha x_\alpha=0$, 
and by $E^a$ the order continuous part of $E$. 

\begin{lemma}\label{E' is o-cont}
{\em (\cite[Theorem~4.59]{AlBu} and \cite[2.4.14]{Mey91})}
TFAE. 
\begin{enumerate}[\em (i)] 
\item 
The norm on $E^\ast$ is order continuous.
\item 
$E^\ast$ is a \text{\rm KB}-space.
\item 
Every disjoint bounded sequence in $E$ is weakly null.
\end{enumerate}
\end{lemma}

\begin{definition}\label{limited subset}
{\em
A bounded subset $C$
\begin{enumerate}[a)]
\item 
of $X$ is {\em limited}, 
if each \text{\rm w}$^\ast$-null  sequence in $X^\ast$ is uniformly null on $C$.
\item
of $X$ is {\em Dunford-Pettis} or a \text{\rm DP} {\em set}
if each weakly null  sequence in $X^\ast$ is uniformly null on $C$.
\item 
of $E$ is {\em L-weakly compact} or an \text{\rm Lwc} {\em set} if each disjoint 
sequence in $\sol(C)$ is norm null.
\item
of $E$ is {\em almost order bounded} if, for each $\varepsilon>0$ there exist $u_\varepsilon\in E_+$ 
satisfying $C\subseteq[-u_\varepsilon,u_\varepsilon]+\varepsilon B_X$.
\end{enumerate}}
\end{definition}

\noindent
It is well known that $B_X$ is not limited unless $\dim(X)<\infty$ (however, $B_{c_0}$ 
is a limited subset of $\ell^\infty$ by Phillip's Lemma \cite[Theorem 4.67]{AlBu}).  For every \text{\rm Lwc} subset $A$ of $E$, 
we have $A\subseteq E^a$ \cite[p.\,212]{Mey91}. The following technical lemma is an extension of the 
Grothendieck theorem.  

\begin{lemma}\label{Lemma 3}
Let $C$ be a bounded subset of $X$ $($or $E$$)$. If, for each $\varepsilon > 0$, there exists
\begin{enumerate}[\em (i)]
\item 
a relatively compact subset $C_\varepsilon \subseteq X$, 
$C\subseteq C_\varepsilon + \varepsilon B_X$, then $C$ is relatively compact.
\item 
a relatively weakly compact $C_\varepsilon\subseteq X$, 
$C\subseteq C_\varepsilon + \varepsilon B_X$, then $C$ is relatively weakly compact.
\item 
a limited $C_\varepsilon\subseteq E$, 
$C\subseteq C_\varepsilon + \varepsilon B_E$, then $C$ is limited.
\item 
an almost limited $C_\varepsilon\subseteq E$, $C\subseteq C_\varepsilon + \varepsilon B_E$, then $C$ is almost limited.
\end{enumerate}
\end{lemma}

\noindent
The item (i) of Lemma~\ref{Lemma 3} is a standard fact. 
For (ii), we refer to~\cite[Theorem 3.44]{AlBu}.
The items (iii) and (iv) are contained in \cite[Lemma 2.2]{EMM13}.

The following two definitions will be permanently used through the paper.

\begin{definition}\label{Main Schur property}
{\em A Banach space $X$ has:
\begin{enumerate}[a)]
\item
the {\em Dunford--Pettis property}  (or $X\in\text{\rm (DPP)}$) if $f_n(x_n)\to 0$ for each 
weakly null $(x_n)$ in $X$ and each weakly null $(f_n)$ in $X^\ast$ \cite[p.\,341]{AlBu};
\item 
the {\em Gelfand--Phillips property} (or $X\in\text{\rm (GPP)}$)
if limited subsets of $X$ are relatively compact \cite{BD}.
\end{enumerate}}
\end{definition}

\noindent
Separable Banach spaces and reflexive Banach spaces are 
in \text{\rm (GPP)} by \cite{BD}.
The closed unit ball of $\ell^\infty$ is almost limited, but 
the closed unit ball of $c_0$ is not almost limited.

\begin{definition}\label{Schur property}
{\em A Banach lattice $E$ has:
\begin{enumerate}[a)]
\item
the {\em dual disjoint \text{\rm w}$^\ast$-property} (or $E\in${\rm DDw}$^\ast${\rm P}) if every disjoint bounded sequence in $E^\ast$ is \text{\rm w}$^\ast$-null. 
\item 
the {\em positive Schur property} (or $E\in\text{\rm (PSP)}$)
if each positive weakly null sequence in $E$ is norm null \cite{Wnuk2013};
\item 
the {\em dual disjoint Schur property} (or $E\in\text{\rm (DDSP)}$)
if each disjoint \text{\rm w}$^\ast$-null sequence in $E^\ast$ is norm null 
\cite[Definition~3.2)]{EMM13};
\item
the {\em property} (d) (or $E\in\text{\rm (d)}$) if, 
for every disjoint $\text{\rm w}^\ast$-null $(f_n)$ in $E^\ast$,
the sequence $(|f_n|)$ is also $\text{\rm w}^\ast$-null \cite[Definition~1]{El};
\item 
the {\em positive Grothendieck property} (or $E\in\text{\rm (PGP)}$) if each positive 
\text{\rm w}$^\ast$-null sequence in $E^\ast$ is weakly null \cite[p.\,760]{Wnuk2013}; 
\item 
the {\em disjoint Grothendieck property} (or $E\in(\text{\rm DGP})$) 
if disjoint \text{\rm w}$^\ast$-null sequences in $E^\ast $ are weakly null.
\end{enumerate}}
\end{definition}

\noindent
The \text{\rm DGP} was studied in \cite{MFMA} under the name ``weak Grothendieck property".
The following assertion illustrates relations of $E\in${\rm DDw}$^\ast${\rm P} to other properties.

\begin{assertion}\label{DDw*P}
For every Banach lattice $E$ the following holds.
$$
   E\in(\text{\rm DGP})\cap(\text{\rm DDw}^\ast\text{\rm P})\Longrightarrow E^{\ast\ast} 
   \text{\it \ is a \text{\rm KB}-space}\Longrightarrow 
   E\in(\text{\rm DDw}^\ast\text{\rm P}).
$$
\end{assertion}

\begin{proof}
Let $E\in(\text{\rm DGP})\cap(\text{\rm DDw$^\ast$P})$ but $E^{\ast\ast}$ 
is not a \text{\rm KB}-space.
Then, by Lemma~\ref{E' is o-cont}, there is a disjoint bounded $(f_n)$ in $E^\ast$
which is not weakly null. However $f_n\stackrel{\text{\rm w}^\ast}{\to}0$
because $E\in(\text{\rm DDw$^\ast$P})$, and hence 
$f_n\stackrel{\text{\rm w}}{\to}0$ as $E\in(\text{\rm DGP})$.
The obtained contradiction proves the first implication.

For the second implication, let $E^{\ast\ast}$ be \text{\rm KB}. By Lemma~\ref{E' is o-cont}, every disjoint $(f_n)$ in $B_{E^\ast}$ 
is weakly null and hence $\text{\rm w}^\ast$-null.  
It follows $E\in(\text{\rm DDw$^\ast$P})$.
\end{proof}

\noindent
Recall that $E$ {\em has sequentially $\text{\rm w}^\ast$-continuous lattice operations} if
$f_n\stackrel{w^\ast}{\to}0$ in $E^\ast$ implies $|f_n|\stackrel{w^\ast}{\to}0$.
Any AM-space with unit has sequentially $\text{\rm w}^\ast$-continuous lattice operations.
The property \text{\rm (d)} is weaker than the sequential 
$\text{\rm w}^\ast$-con\-ti\-nui\-ty of the lattice operations, 
e.g. in $\ell^{\infty}$. It was proved in \cite[Proposition~1.4]{Wnuk2013}, 
that every Dedekind $\sigma$-complete $E$ has the property (d) \cite[Corollary~2.4]{EMM13}, 
and it was observed in \cite[Remark~1.5]{Wnuk2013} 
that $\ell^{\infty}/c_0\in\text{\rm (d)}$
but $\ell^{\infty}/c_0$ is not Dedekind $\sigma$-complete.
The lattice ${\text\rm L}^p[0,1]$, $1\leq p\leq\infty$, has the property~(d), 
whereas $C[0,1] \not\in\text{\rm (d)}$ (Example~\ref{DPSP but not DSP}).    

\medskip
\noindent
The following assertion \cite[Theorem~3.3]{EMM13} describes 
the impact of the property~(d) on dual Schur properties.

\begin{assertion}\label{DSP and (d)} 
The following holds for Banach lattices$:$
\begin{enumerate}[\em (i)]
\item
$(\text{\rm DDSP}) \subseteq (\text{\rm DPSP})$$;$
\item 
$((\text{\rm DDSP})\cap \text{\rm (d)}) = ((\text{\rm DPSP})\cap \text{\rm (d)})$.
\end{enumerate}
\end{assertion}

\noindent
Clearly $((\text{PGP})\cap \text{\rm (d)})\subseteq ((\text{DGP})\cap \text{\rm (d)})$.
For this inclusion, the property~(d) is essential, as $C[0,1]\in\text{\rm (PGP)}\setminus\text{\rm (DGP)}$
by Example~\ref{DPSP but not DSP} below.
We have no example of $E\in(\text{DGP})\setminus(\text{PGP})$.
The next example shows that the property~(d) is essential 
in Assertion~\ref{DSP and (d)}(ii). 

\begin{example}\label{DPSP but not DSP}
{\em Consider the Banach lattice $E=C[0,1]$. Let $(f_n)$ be a positive \text{\rm w}$^\ast$-null sequence in $E^\ast $.
It follows from $\|f_n\|=f_n(\mathbb{1}_{[0,1]})\to 0$ that $C[0,1]\in(\text{\rm DPSP})$, 
and hence $C[0,1]\in(\text{\rm PGP})$.
Take the sequence $(r_n)$ of Rademacher functions on $[0,1]$ and define a   
disjoint \text{\rm w}$^\ast$-null sequence in $E^\ast $ as follows:
$g_n=2^{n+1}\cdot r_{2n}\cdot\mathbb{1}_{\big[\frac{1}{2^{n+1}},\frac{1}{2^n}\big]}$.
Then $(g_n)$ is a disjoint \text{\rm w}$^\ast$-null sequence in $E^\ast $, 
yet $\|g_n\|=1$ for all $n\in {\mathbb N}$.
Therefore $C[0,1]\not\in(\text{DDSP})$, and hence 
by Assertion~\ref{DSP and (d)}, 
$C[0,1]\not\in(\text{d})$.
Note that the set $\{g_n: n\in {\mathbb N}\}$ is not even relatively weakly compact. Indeed, for 
$$
  y:=\sum_{n=1}^\infty r_{2n}\cdot 
  \mathbb{1}_{\big[\frac{1}{2^{n+1}}, \frac{1}{2^n}\big]} \in
  L^\infty[0,1] \subseteq E^\ast,
$$
where $y(g_n)\equiv 1$ for all $n\in {\mathbb N}$.
It also shows $C[0,1]\not\in(\text{DGP})$.}
\end{example}

The following two definitions are permanently used below.

\begin{definition}\label{X to Y}
{\em A bounded operator $T: X\to Y$ is 
\begin{enumerate}[a)]
\item
{\em Dunford--Pettis} (or $T\in\text{\rm DP}(X,Y)$) 
if $T$ takes weakly null sequences to norm null ones;
\item
{\em limited} (or $T\in\text{\rm Lim}(X,Y)$) if $T(B_X)$ is a limited subset of $Y$. 
\end{enumerate}
\begin{enumerate}
\item[c)]
A bounded operator $T:E\to Y$ is 
{\em almost} $\text{\rm M}$-{\em weakly compact} (or $T\in\text{\rm aMwc}(E,Y)$)
if $f_n(Tx_n)\to 0$ for each weakly convergent $(f_n)$ in $Y'$ and disjoint bounded $(x_n)$ in $E$ \cite[Definition~2.2]{BLM18}.
\end{enumerate}
}
\end{definition}

\noindent
For instance, $\text{\rm Id}_{c_0}\in\text{\rm aMwc}(c_0)$.
The canonical embedding of $c_0$ into $\ell^\infty$ is limited, but not compact.
Note also that order intervals of a Banach lattice are not necessarily limited
(e.g., $[-{\mathbb 1},{\mathbb 1}]$ in $c$ is not limited, 
where ${\mathbb 1}$ denotes the constantly one sequence).

\begin{definition}\label{E to Y}
{\em A bounded operator
\begin{enumerate}[a)]
\item 
$T:E\to Y$ is \text{\rm AM}{-compact} ($T\in\text{\rm AMc}(E,Y)$) if $T[-x,x]$ is relatively compact for every $x\in E_+$.
\item 
$T:X\to F$ is {\em semi-compact} ($T\in\text{\rm semi-K}(X,F)$) if it carries bounded subsets of $X$ onto almost order bounded subsets of $F$
(or, equivalently, for each $\varepsilon >0$, there exists $u_\varepsilon\in F_+$ with $T(B_X)\subseteq [-u_\varepsilon,u_\varepsilon]+\varepsilon B_F$).
\item  
$T:X\to F$ is $\text{\rm L}$-{\em weakly compact} ($T\in\text{\rm Lwc}(X,F)$) 
if $T$ carries bounded subsets of $X$ onto \text{\rm Lwc} subsets of $F$.
\item
$T:X\to F$ is {\em almost} $\text{\rm L}$-{\em weakly compact} ($T\in\text{\rm aLwc}(X,F)$) 
if $T$ carries relatively weakly compact subsets of $X$ onto 
$\text{\rm Lwc}$ subsets of $F$ \cite[Definition~2.1]{BLM18}
(or, equivalently, if $T(X)\subseteq F^a$ and $f_n(Tx_n)\to 0$ for each 
w-null $(x_n)$ in $X$ and disjoint $(f_n)$ in $F^\ast$).
\item
$T:X\to F$ is {\em limitedly {\rm L}-weakly compact} ($T\in l\text{\rm -Lwc}(X,F)$)
if $T$ maps limited subsets of $X$ onto  \text{\rm Lwc} subsets of $F$ (or, equivalently,
$T^\ast f_n\stackrel{w}{\to}0$ for each disjoint bounded $(f_n)$ in $F^\ast$).
\end{enumerate}}
\end{definition}

\noindent
The limitedly \text{\rm L(M)wc}-operators were defined in \cite{AEG_duality} and \cite{OM22} (in the second paper they were named weak  \text{\rm L(M)wc}-operators).
It is well known that $\text{\rm semi-K}(X,F)=\text{\rm Lwc}(X,F)$ for every $X$, whenever the norm in $F$ is \text{\rm o}-continuous.
Note that \cite[Theorem~5.63]{AlBu}  yields a method of producing many classes of operators.
For example, if we consider sets of the form $A=T(C)$ with any weakly compact 
subset $C$ of $X$ and $B=B_{F^\ast }$ then we obtain 
almost \text{\rm L(M)wc} operators \cite{BLM18}, whereas the choice of sets of the form 
$A=T[0,x]$ with any order interval $[0,x]$ in $E$ and $B=B_{F^\ast}$ 
produces order \text{\rm L(M)wc} operators \cite{BLM21}.

\medskip
In Section 2 we prove Theorem \ref{Proposition 2.2} that gives condition on bounded operators to be \text{\rm AM}-compact
and Theorem \ref{Proposition 2.4} that gives condition on bounded operators to be limited.
In Section 3, Theorem \ref{Proposition 3.12} provides conditions under which every regular operator is almost limited.
Theorem \ref{c-5.60} of Section 4 gives a tool for lattice approximation for collectively Mwc families of operators.

\medskip
For unexplained terminology and notation, we refer to \cite{AlBu,BD,AEG22,Mey91}.

\section{The effect of \text{\rm w}$^\ast$-continuity of lattice operations in the dual}

In this section, we study how order continuity of norm in $E^\ast$ and continuity of lattice operations
in $E^\ast$ (which imply order intervals being limited in $E$ \cite{EMM13}) effect operators.
The limited order intervals will play an important role in what follows. 
So let us recall that lattice operations are weakly continuous if, for every w-null $(x_n)$ in $E$, 
$|x_n|\stackrel{w}{\to}0$. 
We need the following lemma~\cite[Proposition~2.3\,(1)]{EMM13}.

\begin{lemma}\label{w*cont lo in E}
Let $E$ be a Banach lattice.  The lattice operations of $E^\ast$ are weak$^\ast$ sequentially continuous if and only
 if every almost order bounded subset of $E$ is limited.
 \end{lemma}

\noindent
We begin with a useful summary.

\begin{proposition}\label{Proposition 2.1}
Let $T\in\text{\rm semi-K}(X,F)$. 
\begin{enumerate}[\em (i)]
\item
If $F$ has limited order intervals, then $T\in\text{\rm Lim}(X,F)$.
\item
If $F$ has compact order intervals, then $T\in\text{\rm K}(X,F)$.
\item
If $F$ has \text{\rm o}-continuous norm, then $T\in\text{\rm W}(X,F)$.
\end{enumerate}
\end{proposition}

\begin{proof}
(i) \ 
We need to show that $T(B_X)$ is limited. 
Since $T\in\text{\rm semi-K}(X,F)$ then, for every $\varepsilon > 0$,
there exist $u_\varepsilon\in F_+$ with $T(B_X)\subseteq[-u_\varepsilon,u_\varepsilon]+\varepsilon B_F$. 
As each order interval in $F$ is limited, $T(B_X)$ is limited by Lemma~\ref{Lemma 3}(iii).

\medskip
(ii) \ 
Similar to (i).

\medskip
(iii) \ 
Since $F^a=F$, the order intervals in $F$ are weakly compact.
The rest of the proof is similar to (i).
\end{proof}

\begin{corollary}\label{Corollary 2.2.5}
Let $E^\ast$ be a \text{\rm KB}-space and let $F$ have limited order intervals. The following hold.
\begin{enumerate}[{\em (i)}]
\item 
$\text{\rm DP}_+(E,F) \subseteq\text{\rm Lim}(E,F)$.
\item
If $E\in \text{\rm (DPP)}$, then $\text{\rm W}_+(E,F)\subseteq \text{\rm Lim}(E,F)$.
\end{enumerate}
\end{corollary}

\begin{proof}
(i) \ 
Let $T\in \text{\rm DP}_+(E,F)$. By \cite[Theorem~3.7.10]{Wnuk90}, 
$T\in\text{\rm Mwc}(E,F)$. Thus $T\in\text{\rm semi-K}(E,F)$, and hence
$T\in\text{\rm Lim}(E,F)$ by Proposition~\ref{Proposition 2.1}(i). 

\medskip 
(ii)\ 
Let $T\in \text{\rm W}_+(E,F)$. Then $T\in \text{\rm DP}(E,F)$ by 
\cite[Theorem~5.82]{AlBu}. The rest of the proof is as in~(i).
\end{proof}

\medskip\noindent
Since L-weakly compact, regular M-weakly compact, and positive operators dominated by semi-compact operators 
are semi-compact~(\cite[Theorem~5.71, Theorem~5.72]{AlBu}), they are (almost) limited 
independent of their ranges due to Proposition~\ref{Proposition 2.1}.
A semi-compact operator need not be compact, weakly compact, L- or M-weakly compact. 
The identity operator on $\ell^\infty$ is  semi-compact, but it has 
none of the compactness properties mentioned.

\begin{corollary}\label{Corollary 2.1.1}
Let $F$ have \text{\rm o}-continuous norm and limited order intervals. Then, for every $E$ 
\begin{enumerate}[\em (i)]
\item
each semi-compact operator $T:E\to F$ is compact$;$
\item
each order bounded \text{\rm Mwc} operator  $T:E\to F$ is compact$;$
\item
each \text{\rm Lwc} operator  $T:E\to F$ is compact$;$
\item
if $E^\ast$ has \text{\rm o}-continuous norm, then each order bounded \text{\rm aDP} operator $T:E\to F$ is compact.
\end{enumerate}
\end{corollary}

\begin{proof}
(i)\ 
Let $T\in\text{\rm semi-K}(E,F)$. By Proposition~\ref{Proposition 2.1}, $T\in\text{\rm Lim}(E,F)$, and
so $T(B_E)$ is limited. As $F^a=F$, then $F$ is Dedekind $\sigma$-complete.
Then $F\in\text{\rm (GPP)}$ by~\cite{Buh}. 
Thus $T(B_E)$ is relatively compact, so $T\in\text{\rm K}(E,F)$.

\medskip
(ii) \ 
This follows from (i) in view of \cite[Theorem~5.72]{AlBu}.

\medskip
(iii) \ 
This follows from (i) in view of \cite[Theorem~5.71]{AlBu}.

\medskip
(iv) \ 
If $E^\ast$ is a \text{\rm KB}-space then, by \cite[Proposition~2.2]{AqEl11}, $T\in\text{\rm aDP}(E,F)$ and is order bounded.
Hence $T\in\text{\rm Mwc}(E,F)$ and therefore $T\in\text{\rm semi-K}(E,F)$. By (i),   $T\in\text{\rm K}(E,F)$.
\end{proof}

\medskip\noindent
By \cite[Proposition~2.2]{AqEl11}, if $E^\ast$ is a \text{\rm KB}-space, 
$T\in\text{\rm aDP}(E,F)$ and is order bounded, then $T\in\text{\rm semi-K}(E,F)$. 
Thus it follows easily that if, additionally, $S\in\text{\rm AMc}(F,G)$
then $ST\in\text{\rm K}(E,G)$. Note that not every weakly compact operator is \text{\rm AM}-compact,
e.g., the identity operator $I_{L^2[0,1]}$ is weakly compact but it is not \text{\rm AM}-compact.

\medskip
\noindent
An operator $T:E\to Y$ is called 
\begin{enumerate}[a)]
\item
{\em order limited} (or $T\in \text{\rm o-Lim}(E,Y)$) 
if $T[0,x]$ is limited  for each $x\in E_+$;
\item
{\em order weakly compact} (or $T\in \text{\rm owc}(E,Y)$) 
if $T[0,x]$ is relatively weakly compact for each $x\in E_+$.
\end{enumerate}

\begin{theorem}\label{Proposition 2.2}
Let all order intervals in $E$ be limited. The following holds$:$
\begin{enumerate}[{\em (i)}]
\item
$\text{\rm W}(E,Y)\subseteq \text{\rm AMc}(E,Y)$ for all $Y$.
\item
$\text{\rm owc}(E,Y)\subseteq\text{\rm AMc}(E,Y)$  for all $Y$.
\item
If $E^\ast$ and $F$ both have $\text{\rm o}$-continuous norms, then 
$\text{\rm L}_+(E,F)\subseteq \text{\rm AMc}(E,F)$.
\item
If $E^\ast$ is a $\text{\rm KB}$-space and  $Y$ does not contain $c_0$, then 
$\text{\rm L}(E,Y)\subseteq \text{\rm AMc}(E,Y)$.
\end{enumerate}
\end{theorem}

\begin{proof}
(i) \ 
Let $T\in\text{\rm W}(E,Y)$, and let $[-x,x]\subseteq E$.  
By~\cite[Theorem~5.38]{AlBu},
$T$ factors over a reflexive $Z$ as $T = SR$ with  $R\in \text{\rm L}(E,Z)$ 
and $S\in \text{\rm L}(Z,Y)$.
The set $R[-x,x]$ is limited in $Z$ as $[-x,x]$ is limited in $E$. 
Since $Z\in(\text{\rm GPP})$ by~\cite{BD},   
$R[-x,x]$ is relatively compact in $Z$, and hence $SR[-x,x]$ 
is  relatively compact in $Y$. Thus, $T\in\text{\rm AMc}(E,Y)$.

\medskip
(ii) \ 
Let $T\in\text{\rm owc}(E,Y)$ and let $[-x,x]\subseteq E$. 
Then $T$ admits a factorization  over some  $F=F^a$ as $T = SQ$, 
where $Q: E\to F$ is a lattice homomorphism and $S\in \text{\rm L}(F,Y)$  
by \cite[Theorem~5.58]{AlBu}. 
The set $Q[-x,x]$ is limited as $[-x,x]$ is limited. 
Since $F$ has \text{\rm o}-continuous norm, 
it is Dedekind $\sigma$-complete, hence $F\in(\text{\rm GPP})$ by~\cite{Buh}. 
Therefore $Q[-x,x]$ is relatively compact in $F$, and hence $SQ[-x,x]$ is relatively compact  in $Y$. 
Thus $T\in\text{\rm AMc}(E,Y)$.

\medskip
(iii) \ 
By \cite[Exercise~15, p.\,314]{AlBu}, $T$ has a factorization $T = SR$ over a Banach lattice $G$
with \text{\rm o}-continuous norm. Let $[a, b]\subseteq E$. Then $R[a, b]$ is limited in $G$. 
Since $G\in(\text{\rm GPP})$ (by~\cite{Buh}), $R[a, b]$ is relatively compact in $G$,  
and hence $SR[a, b]$ is relatively compact in $F$. 
Since the interval $[a, b]$ is arbitrary, $T\in  \text{\rm AMc}(E,F)$.

\medskip
(iv) \ 
By \cite[Exercise.\,19, p.\,315]{AlBu}, the operator $T$ admits a factorization over a reflexive Banach lattice $G$ as 
$T = SR$, where $R:E\to G$ and $S:G\to Y$ are bounded. The rest of the proof is similar to proof of~(i) and is omitted.
\end{proof}

\begin{corollary}\label{Corollary 2.2.2 +2.2.3}
Let $E$ have limited order intervals, and let $F$ have \text{\rm o}-continuous norm. Then the following holds.
\begin{enumerate}[\em (i)]
\item
$\text{\rm semi-K}_+(E,F)\subseteq\text{\rm AMc}(E,F)$.
\item
$\text{\rm Lwc}(E,F)$ and $\text{\rm Mwc}(E,F)$ are $\text{\rm AMc}(E,F)$.
\item
Let $T:E\to F$ have \text{\rm o}-continuous operator norm. Then $T\in\text{\rm AMc}(E,F)$.
\end{enumerate}
\end{corollary}

\begin{proof}
(i) \ 
As $T\in\text{\rm semi-K}_+(E,F)$, for each $\varepsilon > 0$, we can find 
$u_\varepsilon\in F_+$ such that $T(B_E)\subseteq [-u_\varepsilon,u_\varepsilon] +\varepsilon B_F$. 
Since $F$ has \text{\rm o}-continuous norm, the interval $[-u_\varepsilon,u_\varepsilon]$ is weakly compact. 
Consequently $T\in\text{\rm W}(E,F)$ 
by Lemma~\ref{Lemma 3}(ii), and $T\in\text{\rm AMc}(E,F)$ due to Proposition~\ref{Proposition 2.2}.

\medskip
(ii) \ 
It follows from (i), as L- and M-weakly compact operators are weakly compact.

\medskip
(iii) \ 
By \cite[Theorem~5.68]{AlBu},  $T\in\text{\rm Lwc}(E,F)$ \cite[Theorem~5.68]{AlBu}.
The rest follows from~(ii). 
\end{proof}


\begin{corollary}\label{Proposition 2.3}
Suppose $E^\ast$ is a \text{\rm KB}-space, $E$ has limited order intervals, 
and $F\in (\text{\rm PSP})$. Then $\text{\rm W}_+(E,F) \subseteq \text{\rm K}(E,F)$.
\end{corollary}

\begin{proof}
Let $T\in \text{\rm W}_+(E,F)$. Then $T\in\text{\rm AMc}(E,F)$ by 
Proposition~\ref{Proposition 2.2}. On the other hand, $T\in\text{\rm Mwc}(E,F)$ by \cite[Theorem~3.3]{CW99}. 
As $T$ is positive, it is semi-compact by \cite[Theorem~5.72]{AlBu}.
Since $F\in(\text{\rm PSP})$, it has \text{\rm o}-continuous norm. As $E^\ast$ and $F$ both
have \text{\rm o}-continuous norms, $T\in\text{\rm K}(E,F)$ by \cite[Theorem~125.5]{Wnuk90} 
because of $T\in\text{\rm semi-K}(E,F)\cap\text{\rm AMc}(E,F)$.
\end{proof}

\medskip\noindent
The next result gives some conditions on an operator to be limited.

\begin{definition}\label{aDPO}{\em
An operator $T:E\to Y$ is called {\em almost Dunford--Pettis} ($T\in\text{\rm aDP}(E,Y)$) 
if $T$ maps disjoint weakly-null sequences to norm-null sequences.}
\end{definition}

\noindent
Note that $\text{\rm Id}_{L^1[0,1]}$ is an almost Dunford--Pettis operator but not  Dunford--Pettis one.

\begin{theorem}\label{Proposition 2.4}
If $E^\ast$ is a \text{\rm KB}-space and $F^\ast$ has 
\text{\rm w}$^\ast$-continuous lattice operations,  
then 
$$
\text{\rm aDP}_+(E,F) \subseteq \text{\rm Lim}(E,F).
$$
If $E^\ast$ has \text{\rm w}$^\ast$-sequentially continuous lattice operations, 
then 
$$
\text{\rm Mwc}(E,F)\subseteq \text{\rm Lim}(E,F).
$$
\end{theorem}

\begin{proof}
Let $T\in\text{\rm aDP}_+(E,F)$. By \cite[Proposition~2.2]{AqEl11},  $T\in\text{\rm semi-K}(E,F)$.
Thus, for each $\varepsilon>0$, there exists some $u_\varepsilon\in F_+$ such that 
$T(B_E )\subseteq [-u_\varepsilon,u_\varepsilon] +\varepsilon B_F$. 
Since $F^\ast$ has \text{\rm w}$^\ast$-continuous lattice operations,
 $F$ has limited order intervals, hence $T(B_E)$ is limited by Lemma~\ref{Lemma 3}(iii).

\medskip
 Let $(f_n)\stackrel{w^\ast}{\to}0$ in $F^\ast$. 
We need to show $\| T^\ast f_n\|\to 0$. It suffices to show $|T^\ast f_n|\stackrel{w^\ast}{\to}0$ 
and $T^\ast (f_n)(x_n)\to 0$ for each norm bounded disjoint sequence $(x_n)$ in $E_+$. 
Since $T$ is bounded, $T^\ast f_n\stackrel{w}{\to} 0$ in $E^\ast$. 
Weak$^\ast$-continuity of lattice operations in $E^\ast$ gives $|T^\ast f_n|\stackrel{w^\ast}{\to} 0$. 
Consider
$$
  |T^\ast (f_n)(x_n)| = |f_n(Tx_n)|\leq\|f_n\| \|Tx_n\|.
$$
Since $(f_n)$ is norm bounded and $T\in \text{\rm Mwc}(E,F)$, the claim follows.
\end{proof}

\noindent
The next example shows a limited operator need not be Mwc.

\begin{example}\label{example1}
{\em
Suppose $E^\ast$ is not a \text{\rm KB}-space and $F\ne \{0\}$. 
Then there exists a positive compact (and hence limited) operator $T:E\to F$
that is not {\em M}-weakly compact. Indeed, pick $y\in F_+\setminus \{0\}$.
Since the norm in $E^\ast$ is not \text{\rm o}-continuous, $E^\ast$ contains 
a sublattice isomorphic to $\ell^\infty$ and  sublattice isomorphic to $\ell^1$;
and $E$ contains a sublattice isomorphic to $\ell^1$ 
(\cite[Theorem~4.51]{AlBu} and \cite[Proposition~2.3.12]{Mey91}). 
Let $i:\ell^1\to E$ be a lattice embedding.
Then $i(\ell^1)$ is the range of a positive projection $P:E\to E$, that is $i(\ell^1)=P(E)$~\cite[Proposition~2.3.11]{Mey91}. 
Define $S:P(E)\to F$ by 
$$
   S\left(i\left(\sum\limits_{k=1}^\infty \beta_k e_k\right)\right):=
   \left(\sum\limits_{k=1}^\infty \beta_k \right)y,
$$
and let $T=SP$. Then $T:E\to F$ is positive. Since $S$ has rank one, 
$T$ is compact (and hence limited). But $T\not\in\text{\rm Mwc}(E,F)$.
This follows from $T(i(e_n))=SP(i(e_n))=S(i(e_n))=y$ and so $\|T(i(e_n))\|\not\to 0$.} 
\end{example}

\noindent
In the opposite direction, we have 

\begin{theorem}\label{Prop1}
If one of the following conditions holds$:$
\begin{enumerate}[\em (i)]
\item 
$\text{\rm Lim}_+(E,F)\subseteq\text{\rm Mwc}(E,F)$ for every $F$$;$ 
\item 
$\text{\rm Lim}(E,F)\subseteq\text{\rm aMwc}(E,F)$ for every $F$$;$
\end{enumerate}
then $E^\ast$ is a {\em KB}-space.
\end{theorem}

\begin{proof}
(i)\ 
Suppose $F=E$ and  $E^\ast$ is not a KB-space. Then, in view of Lemma~\ref{E' is o-cont}
and~\cite[Theorem~4.14]{AlBu}, there exists a disjoint sequence   
$(f_n)\subseteq [0,f]\subseteq E^\ast$ such that $\|f_n\|\not\to 0$.
Pick $y\in E_+$ such that $\|y\| =1$ and a functional $g\in E^\ast$ with $\|g\| = 1$ 
and $g(y) = 1$. 
Define a positive operator $T: E\to E$ by $T(x) = f(x)y$ for $x\in E$. 
As $T$ is compact, $T\in\text{\rm Lim}(E,E)$. 
We claim $T\not\in\text{\rm Mwc}(E,E)$. 
By~\cite[Theorem~5.64]{AlBu},
it is enough to show $T^\ast\not\in\text{\rm Lwc}(E^\ast,E^\ast)$. 
Let us observe that $T^\ast (g) = g(y)f=f$. 
Thus $(f_n)$ is a disjoint sequence in the solid hull of $T^\ast(B_E)$ with $\|f_n\|\not\to 0$. 
Therefore $T^\ast\not\in\text{\rm Lwc}(E^\ast,E^\ast)$ and hence
$T\not\in\text{\rm Mwc}(E,E)$ by~\cite[Theorem~5.64]{AlBu}, violating the assumption
$\text{\rm Lim}_+(E,E)\subseteq\text{\rm Mwc}(E,E)$. A contradiction.

\medskip
(ii)\ 
Under the assumption, $\text{\rm K}(E,F)\subseteq\text{\rm aMwc}(E,F)$, 
hence $E^\ast$ is a \text{\rm KB}-space by \cite[Theorem~2]{EAS}.
\end{proof}

The following theorem deals with the conditions on operators to be \text{\rm (a)Mwc}
(compare with \cite[Proposition 4.5]{AG_aoLwc}).

\begin{theorem}\label{Prop3}
Let $E^\ast$ be a \text{\rm KB}-space. Then the following holds.
\begin{enumerate}[\em (i)]
\item 
If $F\in\text{\rm (GPP)}$ then
$\text{\rm Lim}(E,F)\subseteq\text{\rm Mwc}(E,F)$. 
\item 
$\text{\rm Lim}_+(E,F)\subseteq\text{\rm aMwc}(E,F)$ for every $F$.
\end{enumerate}
\end{theorem}

\begin{proof}
(i)\ 
Let $T\in\text{\rm Lim}(E,F)$ and $(x_n)$ be a disjoint bounded sequence in $E$.
Then, by Lemma~\ref{E' is o-cont}, $x_n\stackrel{w}{\to}0$ and hence $Tx_n\stackrel{w}{\to}0$.
Since $F\in\text{\rm (GPP)}$ and the set $\{Tx_n\}_{n=1}^\infty$ is limited, it is relatively compact.
Then $\|Tx_n\|\to 0$, and hence  $T\in\text{\rm Mwc}(E,F)$.

\medskip
(ii)\ 
 Suppose $T\in \text{\rm Lim}_+(E,F)$.
Then $T\in\text{\rm aLim}_+(E,F)$, and hence 
$T^\ast\in\text{\rm aDP}_+(F^\ast,E^\ast)$~\cite[Theorem~1]{El}. 
Then $T^\ast\in\text{\rm aLwc}_+(F^\ast,E^\ast)$ by \cite[Proposition~2.4]{BLM18},
and thus $T\in\text{\rm aMwc}(E,F)$ in view of~\cite[Theorem~2.5]{BLM18}.
\end{proof}

\noindent
The condition that $E^\ast$ is a \text{\rm KB}-space is essential in Theorem~\ref{Prop3}\,(i) due to
\cite[the example before Theorem~5.62]{AlBu}.

\medskip\noindent
In general, a weakly compact operator need not be limited. For example, $\text{\rm Id}_{\ell^2}$
is weakly compact but not limited. Theorem~5 in \cite{EMM14} shows that if $F$ is $\sigma$-Dedekind complete,
then every almost limited $T:E\to F$ is weakly compact iff $E$ is reflexive or $F=F^a$.

\begin{theorem}\label{Proposition 2.8} 
If  $\text{\rm W}(E,Y)\subseteq\text{\rm Lim}(E,Y)$, then one of the following holds$:$
\begin{enumerate}[\em (i)]
\item 
$E^\ast$ is a \text{\rm KB}-space;
\item 
$Y\in (\text{\rm DPP})$.
\end{enumerate}
\end{theorem}

\begin{proof}
Assume that $E^\ast$ is not a \text{\rm KB}-space
and $Y\not\in\text{\rm (DPP)}$. We construct 
$T\in \text{\rm W}(E,Y)$ and $T\not\in\text{\rm Lim}(E,Y)$. 

Since $E^\ast$ is not a \text{\rm KB}-space, there exists a disjoint sequence $(u_n)$ in $(B_E)_+$ 
which is not weakly null by Lemma~\ref{E' is o-cont}. 
By passing to a subsequence, we may assume that, for some $\theta\in E^\ast_+$ and  $\varepsilon >0$,
$\theta (u_n) > \varepsilon$  for all $n$. 
Then the components $\theta_n$ of $\theta$ 
in the carriers $C_n$ of  $u_n$ 
form  an order bounded disjoint sequence in $E^\ast_+$ such that $\theta_n(u_n) =
\theta (u_n)$ for all $n$ and $\theta_n(u_m) = 0$ if $n\ne m$ by \cite[Theorem~116.3]{ZaII}. 
We define an operator $P:E\to \ell^1$ by 
$$
P(x) =\left(\frac{\theta_k(x)}{\theta(u_k)}\right)_{k=1}^\infty \quad ( x \in E). 
$$
The operator $P$ is well defined as  
$$
\|Px\|= \sum_{k=1}^\infty \left|\frac{\theta_k(x)}{\theta(u_k)} \right|\leq
\sum_{k=1}^\infty \frac{|\theta_k(x)|}{\varepsilon} \leq
\frac{1}{\varepsilon}\sum_{k=1}^\infty \theta_k(|x|)\leq \frac{1}{\varepsilon}\theta(|x|)
$$
holds for each $x\in E$.

As $Y\not\in\text{\rm (DPP)}$, there is a weakly null sequence $(y_n)$ in $Y$ 
and a weakly null (hence \text{\rm w}$^\ast$-null) sequence $(f_n)$ in $Y^\ast$  
such that $(f_n(y_n))\not\to 0$. Passing to a subsequence, 
we may assume that, for some $\epsilon >0$, 
$$
   |f_n(y_n)| > \epsilon \quad\quad  (\forall  n\in {\mathbb N}).
   \eqno(1)
$$ 
We define an operator $S: \ell^1 \to Y$ by $S\alpha=\sum_{k=1}^\infty \alpha_k y_k$.
 By \cite[Theorem~5.26]{AlBu}, $S$ is weakly compact.
 
Then $T = SP\in\text{\rm W}(E,Y)$. 
Suppose $T$ is limited. Then the $\text{\rm w}^\ast$-null sequence $(f_n)$ is  
uniformly null on $T(B_E)$. Since
$$
Tu_n = SPu_n=Se_n=y_n\quad\quad  (\forall  n\in {\mathbb N}),
$$
it follows from~(1) that $|f_n(Tu_n )| > \epsilon$ for all $n$, contradicting uniform
convergence to zero of $(f_n)$ on $T(B_E)$. 
Therefore $T \notin\text{\rm Lim}(E,Y)$, as desired.
\end{proof}

\medskip\noindent
We discuss some composition properties of semi-compact operators. 
\begin{proposition}\label{Proposition 2.14}
Let $F$ have \text{\rm o}-continuous norm, $T\in\text{\rm semi-K}(X,F)$, and $S\in\text{\rm DP}(F,G)$.
Then $ST\in\text{\rm K}(X,G)$.
\end{proposition}

\begin{proof}
Since $T\in\text{\rm semi-K}(X,F)$, for each $\varepsilon>0$, 
there exists $u_\varepsilon\in F_+$ such that 
$T(B_X)\subseteq [-u_\varepsilon,u_\varepsilon] +\varepsilon B_F$. 
Since $F$ has \text{\rm o}-continuous norm, 
$[-u_\varepsilon,u_\varepsilon]$ is relatively weakly compact and hence $T(B_X)$ 
is weakly compact by Lemma~\ref{Lemma 3}(ii). 
In order to show that $ST(B_E)$ is relatively compact, take a sequence $(STx_n)$ 
in $ST(B_X)$. As $T(B_X)$ is relatively weakly compact then, 
by~\cite[Theorem~3.40]{AlBu},
there exists a weakly convergent subsequence $(Tx_{n_k})$.  
Since $S\in\text{\rm DP}(F,G)$, the sequence $(STx_{n_k})$ converges in norm.
As the sequence $(STx_n)$ in $ST(B_X)$ is arbitrary, 
$ST(B_X)$ is compact in $G$, and hence $ST\in\text{\rm K}(X,G)$.
\end{proof}

\noindent
If $E^\ast$ is a {\em KB}-space and  $F=F^a$ then, for each regularly aDP operator $T$
(i.e. there exist $T_1,T_2 \in\text{\rm aDP}_+(E,F)$ such that $T=T_1-T_2$) and 
$S\in\text{\rm DP}(F,G)$, we have $ST\in\text{\rm K}(E,G)$.
This follows from the fact that each positive aDP operator is semi-compact by \cite[Proposition~2.2]{AqEl11}.

\bigskip
\noindent
The square of a semi-compact operator is not compact in general
(e.g., $\text{\rm Id}_c$ is semi-compact but $\text{\rm Id}_c^2=\text{\rm Id}_c$ is not compact). 
This problem was studied in \cite[Theorem~2.4]{AqNZPal2006}. We give another sufficient condition. 

\begin{theorem}\label{prp2.15}
Let $E$ have limited order intervals. 
\begin{enumerate}[{\em (i)}]
\item
If $E$ has \text{\rm o}-continuous norm and $T\in\text{\rm semi-K}_+(E)$ then $T^2\in\text{\rm K}(E)$.
\item
If $T\in\text{\rm semi-K}(E)\cap\text{\rm W}(E)$ then $T^2\in\text{\rm K}(E)$.
\end{enumerate}
\end{theorem}

\begin{proof}
(i) \ 
Let  $T\in\text{\rm semi-K}_+(E)$. Then, for $\varepsilon>0$
there exists $u_\varepsilon\in E_+$ wuth $T(B_E)\subseteq [-u_\varepsilon,u_\varepsilon]+\varepsilon B_E$, 
which yields $T^2(B_{E})\subseteq [-Tu_\varepsilon,Tu_\varepsilon]+\varepsilon\|T\|B_E$.
By the assumption, each $[-Tu_\varepsilon,Tu_\varepsilon]$ is limited, and hence $T^2(B_{E})$ is limited by Lemma~\ref{Lemma 3}(iii). 
By \cite{Buh}, a Dedekind $\sigma$-complete Banach lattice has the Gelfand -- Phillips property 
iff it has $\text{\rm o}$-continuous norm.
Therefore $E\in(\text{\rm GPP})$. Then $T^2(B_{E})$ is relatively compact, and hence 
$T^2\in\text{\rm K}(E)$.

\medskip
(ii) \ 
As $T\in \text{\rm W}(E)$, it factors over a reflexive space $G$ as $T = SR$, 
where $R: E\to G$ and $S: G\to E$ are bounded by~\cite[Theorem~5.38]{AlBu}. 

Let $\varepsilon>0$. Since $T\in\text{\rm semi-K}(E)$, 
there exists $u_\varepsilon\in E_+$ such that $T(B_E)\subseteq [-u_\varepsilon, u_\varepsilon]+\varepsilon B_E$. 
The interval  $[-u_\varepsilon, u_\varepsilon]$  is limited in $E$. 
Applying $R$ to $[-u_\varepsilon, u_\varepsilon]$, we see that $R[-u_\varepsilon, u_\varepsilon]$ is limited in $G$.  
Being reflexive, $G\in(\text{\rm GPP})$, and hence $R[-u_\varepsilon, u_\varepsilon]$ is relatively compact in $G$. 
Therefore, $SR[-u_\varepsilon, u_\varepsilon]$ is relatively compact in $E$ and 
$$
   T^2(B_E)\subseteq SR[-u_\varepsilon, u_\varepsilon]+\varepsilon SR(B_E).
$$
Since $\varepsilon>0$ is arbitrary, Lemma \ref{Lemma 3}(i) implies that $T^2(B_E)$
is relatively compact in $E$, and hence $T^2\in\text{\rm K}(E)$.
\end{proof}


\begin{proposition}\label{Proposition 2.17}
$\text{\rm aLwc}(X,F)\subseteq l\text{\rm -Lwc}(X,F)$.
\end{proposition}

\begin{proof}
Let $T\in \text{\rm aLwc}(X,F)$. As one of defining properties of aLwc operators, we have $T(X)\subseteq F^a$. 
Since $F^a$ has \text{\rm o}-continuous norm, 
$F^a\in\text{(\rm DDw$^\ast$P)}$ so that each limited subset in $F^a$ is \text{\rm Lwc} by~\cite[Theorem~2.6]{CCJ},
and hence $T\in l\text{\rm -Lwc}(X,F)$. 
\end{proof}

\noindent

\begin{proposition}\label{Proposition 2.18}
If $F$ has the Schur property then $l\text{\rm -Mwc}(E,F)\subseteq\text{\rm owc}(E,F)$.
\end{proposition}

\begin{proof}
Let $T\in l\text{\rm -Mwc}(E,F)$. By~\cite[Theorem~5.57]{AlBu}, 
we need to show $\|Tx_n\|\to 0$ for each order bounded disjoint sequence $(x_n)$ in $E$. 
Such a sequence is w-null as it is order bounded and disjoint. Since $T\in l\text{\rm -Mwc}(E,F)$, 
we get $Tx_n\stackrel{{\rm w}}{\to}0$ and so $\|Tx_n\|\to 0$ as $F$ has the Schur property.
\end{proof}

\noindent
A subset $A$ of $E$ is called bi-bounded if $i(A)$ is order bounded in $E^{\ast\ast}$,
where $i:E\to E^{\ast\ast}$ is the canonical embedding. An operator $T:E\to F$
is b-weakly compact if $T(A)$ is relatively weakly compact for each
bi-bounded $A$ in $E$. A similar proof to the given one in Proposition~\ref{Proposition 2.18}
shows that $l\text{\rm -Mwc}(E,F)$ actually contained in the narrower set of b-weakly compact operators~\cite{AAT}.

\section{Almost limited operators}

In this section, we focus on almost limited operators and show they contain many
interesting classes of operators, and the converse problem of almost limited operators
contained in other classes. 

\medskip\noindent
A bounded subset $A$ of a Banach space $X$ is called {\em almost limited} (we write \text{\rm aLim}-set) 
if each disjoint \text{\rm w}$^\ast$-null sequence in $X^\ast$ converges uniformly to zero on $A$. 
Almost limited sets were introduced in~\cite{CCJ}. Let us also note that every order interval in 
$E$ is \text{\rm aLim} iff $E\in {\text\rm (d)}$.
An operator $T: X\to F$ is called {\em almost limited} 
(we write $T\in\text{\rm aLim}(X,F)$) if $T(B_X)$ is an \text{\rm aLim}-set. 
Almost limited sets and operators were studied in \cite{EMM13,CCJ,EMM14}.



\medskip\noindent
It was shown in \cite[Proposition~4.5]{EMM13} that each order
bounded \text{\rm Mwc} operator $T:E\to F$ is almost limited whenever 
lattice operations in $E^\ast$ are \text{\rm w}$^\ast$-se\-qu\-en\-tial\-ly continuous
or $F\in\text{\rm (d)}$. By \cite[Corollary~4]{El}, if lattice operations in $E^\ast$ are 
\text{\rm w}$^\ast$-se\-qu\-en\-tial\-ly continuous 
then $\text{\rm aLim}(X,E)\subseteq\text{\rm Lim}(X,E)$.

\begin{theorem}\label{Proposition 3.1}
Let $E^\ast$ be a \text{\rm KB}-space. The following holds.
\begin{enumerate}[{\em (i)}]
\item
If $F\in\text{\rm (d)}$ or the lattice operations in $E^\ast$ are \text{\rm w}$^\ast$-sequentially continuous
then each order bounded \text{\rm aLwc} operator $T:E\to F$ is \text{\rm aLim}.
\item
If $E$ is reflexive then $\text{\rm aLwc}(E,F)\subseteq\text{\rm aLim}(E,F)$ for every $F$.
\end{enumerate}
\end{theorem}

\begin{proof}
(i)\ 
Take an order bounded $T\in\text{\rm aLwc}(E,F)$. For $T\in\text{\rm aLim}(E,F)$,  
it suffices to show $f_n(Tx_n)\to 0$ for every disjoint \text{\rm w}$^\ast$-null $(f_n)$ in $F^\ast$ 
and disjoint $(x_n)$ in $(B_E)_+$ (\cite[Theorem~2.7]{EMM13}). 
Let $(f_n)$ be disjoint \text{\rm w}$^\ast$-null in $F^\ast$ and let $(x_n)$ be disjoint in $(B_E)_+$.
Since $f_n\stackrel{w^\ast}{\to}0$, it is bounded. As $E^\ast$ is a \text{\rm KB}-space, $x_n\stackrel{w}{\to}0$. 
Since $T\in\text{\rm aLwc}(E,F)$ then $f_n(Tx_n)\to 0$ by \cite[Theorem~2.2]{BLM18}.

\medskip
(ii)\
Let $T\in \text{\rm aLwc}(E,F)$, and let $(f_n)$ be disjoint  w$^\ast$-null in $B_{F^\ast}$.
We need to show that $(f_n)$ is uniformly null on $T(B_E)$. 
Assume in contrary that $\sup\limits_{x\in B_E}|f_n(Tx)|\not\to 0$ as $n\to\infty$.
Then, for some  $\varepsilon >0$, there exists a sequence $(x_{n_k})$ in $B_E$ such that
$$
  |f_{n_k}(Tx_{n_k})|\geq\varepsilon \quad\quad (\forall k\in{\mathbb N}).
  \eqno(2)
$$
$B_E$ is weakly compact because $E$ is reflexive. So,
$x_{n_{k_j}}\stackrel{w}{\to} x\in E$ for some further subsequence $(x_{n_{k_j}})$
by the Eberlein -- Smulian theorem. 
Since $T\in \text{\rm aLwc}(E,F)$, then by \cite[Theorem~2.2]{BLM18}, $f_{n_{k_j}}(Tx_{n_{k_j}})\to 0$
which contradicts~(2).
Therefore $(f_n)$ converges uniformly to zero on $T(B_E)$, and  $T\in \text{\rm aLim}(E,F)$.
\end{proof}

\begin{theorem}\label{Proposition 3.2}
Let $F\in\text{\rm (d)}$. Then the following hold.
\begin{enumerate}[{\em(i)}]
\item
If $F\in\text{\rm (PGP)}$, then $\text{\rm span}(\text{\rm aMwc}_+(E,F))\subseteq{\rm aLim}(E,F)$.
\item
$l\text{\rm-Lwc}(E,F)\cap \text{\rm Mwc}(E,F)\subseteq\text{\rm aLim}(E,F)$ for every $E$.
\end{enumerate}
\end{theorem}

\begin{proof}
(i)\ 
Let $T\in\text{\rm span}(\text{\rm aMwc}_+(E,F))$. We may suppose $T$ is positive. 
For proving $T\in\text{\rm aLim}(E,F)$,  
it suffices to show $f_n(Tx_n)\to 0$ for every disjoint \text{\rm w}$^\ast$-null $(f_n)$ in $F^\ast_+$ 
and disjoint $(x_n)$ in $(B_E)_+$ (\cite[Theorem~2.7]{EMM13}).
Let $(f_n)$ be disjoint \text{\rm w}$^\ast$-null in $F^\ast_+$ and  $(x_n)$ be disjoint in $(B_E)_+$.
As $F\in\text{\rm (PGP)}$,  $f_n\stackrel{w}{\to}0$. By Definition~\ref{E to Y}\,d),
$f_n(Tx_n)\to 0$. Consequently $T\in\text{\rm aLim}(E,F)$.

\medskip
(ii)\ 
Let $T\in l\text{\rm-Lwc}(E,F)\cap \text{\rm Mwc}(E,F)$.
We need to show $T(B_E)$ is almost limited.
It suffices to show $f_n(Tx_n)\to 0$ 
for every disjoint $(x_n)$ in $(B_E)_+$ and every 
disjoint  w$^\ast$-null $(f_n)$ in $B_{F^\ast}$
\cite[Theorem~2.7]{EMM13}. Let $(x_n)$ be disjoint in $(B_E)_+$ and $(f_n)$ be 
disjoint w$^\ast$-null in $B_{F^\ast}$.
Since $T\in l\text{\rm-Lwc}(E,F)$, then $(T^\ast f_n)\stackrel{w^\ast}{\to}0$ in $E^\ast$
by \cite[Theorem~2.7]{AEG_duality}. As $f_n\stackrel{w^\ast}{\to}0$, $\|f_n\|\le M$ for all $n$.
Since $T\in\text{\rm Mwc}$ then $\|Tx_n\|\to 0$, and hence 
$$
    |f_n(Tx_n)|\leq \|f_n\| \cdot \|Tx_n\|\leq M\cdot \|Tx_n\|
$$
implies $f_n(Tx_n)\to 0$ as desired. 
\end{proof}

\medskip\noindent
The following presents conditions for almost limited operators to be L-weakly compact.

\begin{theorem}\label{Proposition 3.6}
The following are equivalent.
\begin{enumerate}[\em (i)]
\item
$F$ has \text{\rm o}-continuous norm.
\item
For every $X$, $\text{\rm aLim}(X,F)\subseteq\text{\rm Lwc}(X,F)$.
\item
For every $E$, each  positive rank one operator $T:E\to F$ is $\text{\rm Lwc}$.
\end{enumerate}
\end{theorem}

\begin{proof}
(i)$\Longrightarrow$(ii):
Let $T\in\text{\rm aLim}(X,F)$. Then $T(B_X)$ is an \text{\rm aLim}-set in $F$. 
By \cite[Theorem~2.6]{CCJ}, $T(B_X)$ is an \text{\rm Lwc}-set, and hence $T\in\text{\rm Lwc}(X,F)$.

\medskip
(ii)$\Longrightarrow$(iii): It is trivial.

\medskip
(iii)$\Longrightarrow$(i):
Since, for each $y\in F_+$, there exists a positive rank one operator $T:E\to F$ such that
$T(E)=\text{\rm span}(y)$, then $y\in F^a$ by \cite[Theorem~5.66]{AlBu}, and hence
$F$ has \text{\rm o}-continuous norm.
\end{proof}

\noindent
Order continuity of the norm in $F$ is essential above as the identity of $\ell^\infty$ is
almost limited but it is neither M- or L-weakly compact.  

\begin{theorem}\label{Proposition 3.18} 
Let $E$ be an \text{\rm AL}-space and let $F$ have \text{\rm o}-continuous norm.
Then each $T\in\text{\rm aLim}(E,F)$ has a modulus $T$ which is almost limited.
\end{theorem}

\begin{proof}
As $F^a=F$ and $T$ is almost limited, then $T$ is Lwc. 
Hence $T$ has the modulus $|T|$ that is Lwc by \cite[Theorem~2.4]{CW99}. 
Therefore $|T|$ is almost limited by \cite[Theorem~2.6]{CCJ}.
\end{proof}

\noindent
However \cite[Theorem~2.2]{CW99} shows there is a regular compact (hence limited) operator  
from $E = L^2[0,1]$ to $F = c(L^2[0,1])$ which is both M- and L-weakly compact whose modulus does not exist. 
The identity of $\ell^\infty$ is almost limited but it is neither M- nor L-weakly compact.

\begin{theorem}\label{Proposition 3.10}
Suppose $E$ is infinite dimensional and has \text{\rm o}-continuous norm, and $F\in\text{\rm (d)}$. 
If $\text{\rm aLim}(E,F)\subseteq \text{\rm Mwc}(E,F)$, 
then $E^\ast$ and $F$ both have order continuous norms.
\end{theorem}

\begin{proof}
If $\text{\rm aLim}(E,F)\subseteq \text{\rm Mwc}(E,F)$ then $\text{\rm K}(E,F)\subseteq \text{\rm Mwc}(E,F)$.
The fact that the norm in $E^\ast$ is o-continuous follows from \cite[Theorem~2]{EAS}.

Suppose now the norm of $F$ is not \text{\rm o}-continuous. 
Then there exist some $u\in F_+$ and a disjoint sequence $(u_n)$ in $[0,u]$ 
such that $\|u_n\|\not\to 0$. 
We may assume $\|u_n\| = 1$  for all $n$. Since $\dim(E)=\infty$, 
there exists a disjoint normalized sequence $(x_n)$ in $E_+$.
By \cite[Theorem~116.3\,(iii)]{ZaII},  there exists a disjoint sequence 
$(g_n)$ in $(B_{E^\ast})_+$ such that $\frac{1}{2}\leq g_n(x_n) \leq 1$ for all $n$ and
$g_n(x_m) = 0$ for $n\ne m$.
Since $E$ has \text{\rm o}-continuous norm, then $g_n\stackrel{w^\ast}{\to} 0$ by \cite[Theorem~116.2]{ZaII}. 

Define positive operator $R: E\to c_0$  by $Rx = (g_k(x))_{k=1}^\infty$. 
Clearly $R(B_E) \subseteq B_{c_0}$. 
Define positive operator $S: c_0\to F$ by  $S\alpha = \sum_n \alpha_n u_n$. 
Then $S(B_{c_0})\subseteq [-u,u]$, and hence $SR(B_E)\subseteq [-u,u]$. 
As order intervals in $F$ are almost limited by the (d)-property of $F$,  $SR(B_E)$ is almost limited, and 
$SR\in\text{\rm aLim}(E,F)$. 
Since $SR(x_n) = S(g_n(x_n)\cdot e_n)=g_n(x_n)\cdot u_n$, then $\|SR(x_n)\|\geq |g_n(x_n)|\geq\frac{1}{2}$
for all $n$. Therefore $SR\not\in\text{\rm Mwc}(E,F)$. A contradiction.
\end{proof}

\medskip\noindent
An operator $T:E\to F$ is called {\em positively limited} if $\|T^\ast f_n\|\to 0$ for each \text{\rm w}$^\ast$-null $(f_n)$ in $F_+^\ast$.
We denote the set of such operators by $\text{\rm p-Lim}(E,F)$. We need the following lemma~\cite[Theorem~2.2]{ArdaChen23}. 

\begin{lemma}\label{Lemma33}
$T:E\to F$ is positively limited iff $T^\ast(y^\ast_n)(x_n)\to 0$ for each 
$y_n^\ast\stackrel{w^\ast}{\to}0$ in $F^\ast_+$
 and every disjoint bounded $(x_n)$ in $E_+$.
\end{lemma}

\begin{theorem}\label{Proposition 3.12}
Suppose $E^\ast$ and $F$ both have \text{\rm o}-continuous norms, $F$ has limited order intervals,
and each disjoint bounded sequence of $E$ is order bounded. Then each regular operator $T:E\to F$ is almost limited.
\end{theorem}

\begin{proof}
Suppose $T:E\to F$ is positive. $T$ admits a factorization over a Banach lattice $G$ 
with $G$ and $G^\ast$ both have \text{\rm o}-continuous norms, say $T = SR$, 
where $R$ and $S$ are both positive as in~\cite[Problem~15 on page~314]{AlBu}. 
We claim $R$ is positively limited. Let $(z_n^\ast)$ be a positive \text{\rm w}$^\ast$-null sequence in $G^\ast$ 
and $(x_n)$ be a positive disjoint bounded sequence in $E$. 
By the assumption, there exists $x$ with $0\leq x_n\leq x$ for each $n$. As $R$ is positive, $0\leq Rx_n \leq Rx$. 
Thus $0\leq R^\ast(z_n^\ast)(x_n)\leq z_n^\ast(Rx_n)\leq z_n^\ast(Rx)$ 
which shows $R^\ast(z_n^\ast)(x_n)\to 0$. By Lemma~\ref{Lemma33}, we conclude that $R$ is positively limited.

Now we show $T=SR\in\text{\rm aLim}(E,F)$. 
Let  $(u_n^\ast)$ be disjoint \text{\rm w}$^\ast$-null in $F^\ast$. 
Since $F^\ast$ has \text{\rm w}$^\ast$-continuous lattice operations, $|u_n^\ast|\stackrel{w^\ast}{\to}0$ in $F^\ast$.  
As $R$ is positively limited, $\|R^\ast S^\ast |u^\ast_n|\|\to 0$.  
Since $R^\ast |S^\ast u^\ast_n|\leq R^\ast S^\ast |u^\ast_n|$, we have
$$
|R^\ast S^\ast u^\ast_n| \leq R^\ast |S^\ast u^\ast_n| \leq  R^\ast S^\ast |u^\ast_n|.
$$
Resulting that $\|R^\ast S^\ast u^\ast_n\| = \|T^\ast u_n\|\to 0$. It follows $T\in\text{\rm aLim}(E,F)$.
\end{proof}

\noindent
It is worth noting that both conditions on $E$ in Theorem~\ref{Proposition 3.12} can be replaced by a single condition that
$E$ is an \text{\rm AM}-space in which every disjoint bounded sequence is order bounded. Indeed, by the Abramovich result~\cite{AB78},
if $E$ is not isomorphic to an \text{\rm AM}-space then $E$ contains a positive disjoint normalized sequence $(x_n)$ such that $\|\vee_{n=1}^{m}x_n\|\to\infty$,
and hence such a sequence $(x_n)$ can not be order bounded \cite{AB78}. Thus, $E$ must be (isomorphic to) an \text{\rm AM}-space, 
and hence the norm in $E^\ast$ is \text{\rm o}-continuous.
Also note that this condition on $E$ is weaker than the condition that $E$ is an \text{\rm AM}-space with unit, as the 
space $C_{\omega}(\Gamma)$ of real-valued bounded countably supported functions on an uncountable set $\Gamma$ still satisfies the condition 
that every disjoint bounded sequence is order bounded, yet $C_{\omega}(\Gamma)$ has not got even a weak unit.

\begin{theorem}\label{Proposition 3.13}
Let $T:E\to E$ be a positive operator  in the ideal center of $E$. Suppose $E\in\text{\rm (DPSP)}$. Then $T$ is positively limited.
\end{theorem}

\begin{proof} 
For simplicity suppose $0\leq T\leq I$. Let $(f_n)$ be a positive \text{\rm w}$^\ast$-null sequence in $E^\ast$. 
Then  $0\leq T^\ast \leq I$. Hence $0\leq T^\ast f_n \leq f_n$ for all $n$. 
Since $E\in\text{\rm (DPSP)}$,  $\|f_n\|\to 0$ hence  $T^\ast f_n\to 0$.
\end{proof}

It is easy to see that $\text{\rm p-Lim}(E,F)$ is a vector space.  We omit the proof of the next easy fact.

\begin{proposition}\label{p-Lim-closed}
{\em Suppose $\|T_n-T\|\to 0$ and $T_n\in\text{\rm p-Lim}(E,F)$ for all $n$. Then $T\in\text{\rm p-Lim}(E,F)$.}
\end{proposition}

\noindent
The following assertion (see, \cite[Theorem 2.4]{AEG_env} and \cite[Proposition 1.10]{EG23aff}).

\begin{assertion}\label{enveloping norm}
{\em
If $\text{\rm P}$ is a norm closed subspace of $\text{\rm L}(E,F)$ then $\text{\rm span}\bigl(\text{\rm P}_+\bigl)$ 
is complete under the enveloping norm $ \|T\|_{\text{\rm r-P}}=\inf\{\|S\|:\pm T\le S\in\text{\rm P}\}$}
\end{assertion}

\noindent
together with Proposition \ref{p-Lim-closed} gives the following result

\begin{corollary}\label{r-d-Lim}{\em
For every $E$ and $F$, $\text{\rm span}\bigl(\text{\rm p-Lim}_+(E,F)\bigl)$ is a Banach space under the enveloping norm.}
\end{corollary}

\noindent
It was shown in \cite[Theorem~1]{EAS} that $\text{\rm K}(X,F)\subseteq\text{\rm aLwc}(X,F)$ for all $X$ 
iff $F$ has \text{\rm o}-continuous norm.
We extend this to almost limited operators. 

\begin{theorem}\label{Proposition 3.14}
$\text{\rm aLim}(X,F)\subseteq\text{\rm aLwc}(X,F)$ iff $F$ has \text{\rm o}-continuous norm.
\end{theorem}

\begin{proof}
The necessity follows from \cite[Theorem~1]{EAS} as $\text{\rm K}(X,F)\subseteq\text{\rm aLim}(X,F)$.

For the sufficiency, suppose in contrary that $\text{\rm aLim}(X,F)\subseteq\text{\rm aLwc}(X,F)$ 
but the norm in $F$ is not \text{\rm o}-continuous. 
Then there exists $y_0 \in F\setminus F^a$. We choose $x_0 \in X$ and $h \in X^\ast$ such that 
$h(x_0) =\| x_0\|=1$ and define $Tx = h(x)y_0$.  The operator $T$ is almost limited as it is compact. 
However $T\not\in\text{\rm aLwc}(X,F)$ as $\{Tx_0\}=\{y_0\}$ is not an \text{\rm Lwc}-set.
A contradiction, as an almost Lwc operator must map $X$ into the order continuous part of $F$.
\end{proof}

\noindent
By \cite[Theorem~2]{EAS}, 
$\text{\rm K}(E,X)\subseteq \text{\rm aMwc}(E,X)$ iff $E^\ast$ is a \text{\rm KB}-space.
We extend this result as follows.

\begin{proposition}\label{Proposition 3.15}
If $E^\ast$ is a \text{\rm KB}-space then $\text{\rm aLim}_+(E,F)\subseteq \text{\rm aMwc}(E,F)$. 
\end{proposition}

\begin{proof}
Assume that $E^\ast$ is a \text{\rm KB}-space and $T\in\text{\rm aLim}_+(E,F)$. 
Then $T^\ast\in\text{\rm aDP}_+(F^\ast, E^\ast)$ \cite[Theorem~1]{El}. 
Consequently $T^\ast\in\text{\rm aLwc}(F^\ast, E^\ast)$ by \cite[Proposition~2.4]{BLM18}. 
Then $T\in\text{\rm aMwc}(E,F)$ by \cite[Theorem~2.5]{BLM18}.
\end{proof}

\begin{theorem}\label{Proposition 3.17}
Suppose $E^\ast$ and $F$ both have \text{\rm o}-continuous norms 
then $\text{\rm span}(\text{\rm aLim}_+(E,F))\subseteq\text{\rm Mwc}(E,F)$.
\end{theorem}

\begin{proof}
Let $T\in \text{\rm span}(\text{\rm aLim}_+(E,F))$. Suppose $T\geq 0$.  Let $(x_n)$ be a disjoint sequence in $B_E$, 
we show $\|Tx_n \|\to 0$. Since $E^\ast$ is a \text{\rm KB}-space, $x_n\stackrel{w}{\to}0$ in $E$. 
Since $(x_n)$ is disjoint, $|x_n|\stackrel{w}{\to}0$~\cite[Proposition~1.3]{Wnuk2013}, and hence 
$T|x_n |\stackrel{w}{\to}0$ in $F$. As $T$ is positive, $|Tx_n | \leq  T |x_n |$ 
and $|Tx_n |\stackrel{w}{\to}0$ in $F$. 
It remains to show $f_n (Tx_n )\to 0$ for each disjoint bounded sequence $(f_n)$ in $F^\ast$. 
This follows  from $|f_n(Tx_n)|\leq \|T^\ast f_n\| \|x_n \|$, since $\|T^\ast f_n\|\to 0$ as $T$
is almost limited.
\end{proof}

\noindent
We continue with a necessary conditions for positive weakly compact operators to be almost limited~Theorem~\ref{Proposition 2.8}.

\begin{proposition}\label{Proposition 3.16}
Suppose $E=E^a$ and $\text{\rm W}(X,E)\subseteq\text{\rm aLim}(X,E)$. 
Then  $E$ is a \text{\rm KB}-space or $X\in\text{\rm(DPP)}$.
\end{proposition}

\begin{proof}
Suppose $E$ is not a  \text{\rm KB}-space. We show $X\in\text{\rm(DPP)}$. 
Since $E$ is not a  \text{\rm KB}-space, it contains a lattice copy (say $F$) of $c_0$. 
Then $F$ is a closed linear span of some disjoint sequence $(y_n)$ in $E_+$ such 
that $d = \inf_n\| y_n\|>0$ and $\sup_m \|\sum_{i=1}^m y_i\| <\infty$ by \cite[Proposition~0.5.1]{Wnuk90}. 
The operator 
$$
 S\beta = \sum_{k=1}^\infty \beta_k y_k \quad\quad (\beta=(\beta_k)_k\in c_0)
$$ 
is a lattice homomorphism from $c_0$ onto $F$. 
There exists a positive disjoint $(f_n)$ in  $B_{E^\ast}$ such that  $f_n(y_n) \geq d$ 
for all $n$ and $f_n(y_m) = 0$ if $n \ne m$. 
Take $B_X\ni x_n\stackrel{w}{\to}0$ and $B_{X^\ast} \ni x^\ast_n\stackrel{w}{\to} 0$. 
Define $R : X\to  c_0$ by 
$$
  Rx = (x^\ast_k(x))_k.
$$ 
Then $R\in\text{\rm W}(X,c_0)$ by \cite[Theorem~5.26]{AlBu}. 
Therefore $T = SR\in\text{\rm W}(X,E)$ and hence $T\in\text{\rm aLim}(X,E)$ by the hypothesis. 
Since $E$ has \text{\rm o}-continuous norm, the bounded disjoint sequence $(f_n)$ is \text{\rm w}$^\ast$-null 
and therefore $\|T^\ast f_n\|\to 0$ as $T\in\text{\rm aLim}(X,E)$. 
Since $(x_n)$ is bounded and $|f_n(Tx_n)| \leq\| x_n\| \cdot\| T^\ast f_n\|$, these yield $|f_n(Tx_n)|\to 0$. 
Thus, it follows from
$$
 | f_n(Tx_n)| = 
|f_n(S(R(x_n)))|= \left| f_n\left(\sum_{k=1}^\infty [R(x_n)]_k y_k\right)\right|=
$$
$$
|f_n ([R(x_n)]_n y_n)|=|f_n ([(x^\ast_k(x_n))_k]_n y_n)|=
|f_n (x^\ast_n(x_n) y_n)|=
$$
$$
|x^\ast_n(x_n) f_n (y_n)| =
|x^\ast_n(x_n)|\cdot | f_n (y_n)|\geq |x^\ast_n(x_n)|\cdot d
$$
that $x^\ast_n(x_n)\to 0$, and hence $X\in\text{\rm(DPP)}$.
\end{proof}

\begin{corollary}\label{Corollary 3.17.1}
$\big[\text{\rm W}(X,c_0)\subseteq\text{\rm aLim}(X,c_0)\big]\Longrightarrow X\in\text{\rm (DPP)}$.
\end{corollary}

\section{Collectively limited and almost limited operator families}

Collectively qualified families of operators were studied recently in \cite{Em24coll_order,Em24coll,Em24coLevi,AEG_collect1}. Some of the results given 
previously in the present paper for individual operators have collective versions as we show below.
At the expense of few more definitions~\cite{Em24coll,Em24coll_order,AEG_collect1}, we can give more general results extending results 
obtained in previous sections.

\begin{definition}\label{collect}
{\em A norm bounded family ${\cal T}$ of operators is
\begin{enumerate}[a)]
\item
{\em collectively compact} (or ${\cal T}\in\text{\bf K}(X,Y)$) if ${\cal T}B_X=\bigcup\limits_{T\in{\cal T}}T(B_X)$ is relatively compact in $Y$.
\item
{\em collectively limited} (or ${\cal T}\in\text{\bf Lim}(X,Y)$) if ${\cal T}B_X$ is limited in $Y$.
\item
{\em collectively weakly compact} (or ${\cal T}\in\text{\bf W}(X,Y)$) if ${\cal T}B_X$ is relatively weakly compact in $Y$.
\item
{\em collectively} Lwc (or ${\cal T}\in\text{\bf Lwc}(X,F)$) if ${\cal T}B_X$ is an \text{\rm Lwc} set in $F$.
\item
{\em collectively} Mwc (or ${\cal T}\in\text{\bf Mwc}(E,Y)$) if 
$\sup\limits_{T\in{\cal T}}\|Tx_n\|\to 0$ for every disjoint $(x_n)$ in $B_E$.
\item
{\em collectively semi-compact} (or ${\cal T}\in\text{\bf semi-K}(X,F)$) if,
for each $\varepsilon >0$ there exists $u\in F_+$ such that ${\cal T}B_X\subseteq [-u,u]+\varepsilon B_F$.
\item
{\em collectively order bounded} (or ${\cal T}\in\text{\bf L}_{ob}(E,F)$) if, for each $a\in E_+$, there exists $b\in F_+$ with  
$\bigcup\limits_{T\in{\cal T}}T[-a,a]\subseteq [-b,b]$.
\item
{\em collectively almost Dunford--Pettis} (or ${\cal T}\in \text{\bf aDP}(E,F)$) if, for each disjoint $x_n \stackrel{w}{\to}0$ on $E$,  
$\sup\limits_{T\in{\cal T}} \|Tx_n\|\to 0$.
\end{enumerate}}
\end{definition}

\noindent
We give an example of collectively limited set consisting of compact operators that is not 
collectively compact.

\begin{example}
{\em 
Let $E=\ell^2$ and $F=\ell^\infty$. Define operators $T_n:E\to F$ by
$T_n(x)=x_n\cdot {\bf e}_n$. Then ${\cal T}=\{T_n: n=1, \dots, \infty\}\subseteq \text{\rm K}(\ell^2,\ell^\infty)$.
The family ${\cal T}$ is not collectively compact since $\bigcup_{n=1}^\infty T_n(B_{\ell^2})$ contains the sequence ${\bf e}_n$
that is not a convergent sequence in $\ell^\infty$.
However, ${\cal T}$ is collectively limited by the Phillips lemma since
$\bigcup_{n=1}^\infty T_n(B_{\ell^2})\subseteq B_{c_0}$.
}
\end{example}

\noindent
Let ${\cal T}$ be a norm bounded family of operators. The following proposition 
gives a method of producing collectively qualified families. We omit its straightforward proof.

\begin{proposition}
Let ${\cal T}\subseteq\text{\rm L}(X,F)$ be norm bounded and $S\in\text{\rm L}(F,G)$. The following is true.
\begin{enumerate}[\em (i)]
\item 
If $S\in\text{\rm K}(F,G)$ then ${\cal T}\circ S\in\text{\bf K}(E,G)$.
\item 
If $S\in\text{\rm W}(F,G)$ then ${\cal T}\circ S \in \text{\bf W}(E,G)$.
\item 
If $S\in\text{\rm Lim}(F,G)$ then ${\cal T}\circ S \in \text{\bf Lim}(E,G)$.
\item 
If $S\in\text{\rm aLim}(F,G)$ then ${\cal T}\circ S \in \text{\bf aLim}(E,G)$.
\item 
If $S\in\text{\rm Lwc}(F,G)$ then ${\cal T}\circ S \in \text{\bf Lwc}(E,G)$.
\item 
If ${\cal T} \in \text{\bf semi-K}(E,F)$ and $S\in\text{\rm Lwc}(F,G)$ then ${\cal T}\circ S \in \text{\bf K}(E,G)$.
\item 
If ${\cal T} \in \text{\bf semi-K}(E,F)$ and $S\in\text{\rm owc}(F,G)$ then ${\cal T}\circ S \in \text{\bf W}(E,G)$.
\end{enumerate}
\end{proposition}

 
\medskip
\noindent
The following result is a consequence of Lemma~\ref{Lemma 3}.

\begin{proposition}\label{col_26.08} 
For every $E$ the following hold$:$
\begin{enumerate}[\em (i)]
\item
If $F=F^a$ then $\text{\bf semi-K}(E,F)\subseteq\text{\bf W}(E,F)$.
\item
If $F=F^a$ then $\text{\bf semi-K}(E,F)\subseteq\text{\bf Lwc}(E,F)$.
\item
If $F$ is discrete and $F=F^a$ then $\text{\bf semi-K}(E,F)\subseteq\text{\bf K}(E,F)$.
\item
If $F$ has limited order intervals, then $\text{\bf semi-K}(E,F)\subseteq\text{\bf Lim}(E,F)$.
\end{enumerate}
\end{proposition}

\begin{proof}
(i)\ 
If ${\cal T}\in\text{\bf semi-K}(X,F)$ then, for every $\varepsilon>0$ there exists
$u_\varepsilon\in E_+$ such that $\bigcup\limits_{T\in{\cal T}}T(B_E)\subseteq [-u_\varepsilon,u_\varepsilon]+\varepsilon B_F$.
Since  $F$ has order continuous norm, $[-u,u]$ is weakly compact~\cite[Theorem~4.9]{AlBu}.
The rest follows from Lemma~\ref{Lemma 3}\,(ii).

\medskip
(ii) \
The statement follows from the well known fact that in a Banach lattice with order continuous norm,
every almost order bounded set is L-weakly compact.   

\medskip
(iii) \ 
By the assumption, each interval $[-u,u]$ is norm compact.
The rest follows from Lemma~\ref{Lemma 3}\,(i).

\medskip
(vi) \ 
It follows from Lemma~\ref{Lemma 3}\,(iii).
\end{proof}

\noindent
The following theorem is an important tool in lattice approximation for collectively Mwc families of operators.

\begin{theorem}\label{c-5.60}
Let ${\cal T}\in\text{\bf Mwc}(E,F)$ be norm bounded. For each $\varepsilon>0$, there exists $u\in E_+$ such that
$\sup\limits_{T\in{\cal T}} \|T\big[(|x|-u)^+\big]\|\le\varepsilon$ holds for all $x\in B_E$.
\end{theorem}

\begin{proof}
Let ${\cal T}\in\text{\bf Mwc}(E,F)$. Take $\varepsilon>0$, $A=B_E$, and consider $\rho(x)=\sup\limits_{T\in{\cal T}}\|Tx\|$.
The seminorm $\rho$ is norm continuous as $\rho(x)\le\|T\|\|x\|\le K\|x\|$ for all $T\in{\cal T}$ and some $K>0$.
If $(x_n)$ is a disjoint sequence $(x_n)$ in $B_E$ then $\rho(x_n)\to 0$ by Definition \ref{collect}.
By \cite[Theorem~4.36]{AlBu}, there exists $u\in E_+$ such that $\sup\limits_{T\in{\cal T}} \|T[(|x|-u)^+]\|=\rho(\text{\rm Id}_E(|x|-u)^+)\le\varepsilon$
for $x\in B_E$.
\end{proof}

\begin{corollary}\label{prop.7_20.Jul}
$\text{\bf L}_{ob}(E,F)\cap\text{\bf Mwc}(E,F)\subseteq\text{\bf semi-K}(E,F)$.
\end{corollary}

\begin{proof}
Let ${\cal T}\in\text{\bf L}_{ob}(E,F)\cap\text{\bf Mwc}(E,F)$ and $\varepsilon>0$. 
Since ${\cal T}\in\text{\bf L}_{ob}(E,F)$, \cite[Theorem 2.1]{Em24coll_order} implies 
that ${\cal T}$ is norm bounded. As ${\cal T}\in\text{\bf Mwc}(E,F)$, by Theorem \ref{c-5.60} there exists $u\in E_+$ with
$$
   \sup\limits_{T\in{\cal T}} \|T(|x|-u)^+\|\le\varepsilon \quad \text{\rm for all} \quad x\in B_E.
  \eqno(3)
$$ 
If $y\in B_E^+$ and $T\in{\cal T}$ then $Ty = T(y\wedge u) +T\big((y-u)^+\big) \in T[0,u]+\varepsilon B_F$ by (3). Therefore,
for every $x\in B_E$ and $T\in{\cal T}$ we have
$$
  Tx = T(x^+) - T(x^-)\in T[0,u]-T[0,u]+2\varepsilon B_F\subseteq T[-u,u]+2\varepsilon B_F.
$$
Thus, $\bigcup\limits_{T\in{\cal T}}T(B_E)\subseteq [-u,u]+2\varepsilon B_F$ proving $\cal T$ is collective semicompact. 
\end{proof}

\begin{example}\label{ex col aDP}
{\em Consider $T:E\to F$ be a positive almost Dunford--Pettis operator.
Then $[0,T]\in \text{\bf aDP}(E,F)$.
Indeed, let $(x_n)$ be a disjoint w-null sequence in $E$. 
Then $(|x_n|)$ is also disjoint and w-null.
Hence $\|T|x_n|\|\to 0$. As $|Sx_n|\leq S|x_n|\leq T|x_n|$
for each $S\in [0,T]$, we have $\|Sx_n\|\leq  \|T|x_n|\|$
and so $\sup\limits_{S\in[0,T]} \|Sx_n\|\leq  \|T|x_n|\|\to 0$.}
\end{example}

\noindent
The following is a collective version of \cite[Proposition~2.1]{AqEl11}.

\begin{proposition}\label{prop.8_20.Jul}
If $E'$ is a $\text{\rm KB}$-space then\\ $\text{\bf L}_{ob}(E,F)\cap\text{\bf aDP}(E,F)\subseteq\text{\bf semi-K}(E,F)$.
\end{proposition}

\begin{proof}
Since ${\cal T}\in\text{\bf L}_{ob}(E,F)$ then ${\cal T}$ is norm bounded by \cite[Theorem 2.1]{Em24coll_order}.
In view of Corollary~\ref{prop.7_20.Jul}, 
it suffices to show ${\cal T}\in\text{\bf Mwc}(E,F)$. 
Let $(x_n)$ be a disjoint sequence in $B_E$. 
Since $E'$ has order continuous norm, $x_n \stackrel{w}{\to}0$ in $E$.
Therefore, $\sup\limits_{T\in{\cal T}} \|Tx_n\|\to 0$ as ${\cal T} \in\text{\bf aDP}(E,F)$.
Thus ${\cal T} \in\text{\bf Mwc}(E,F)$. 
\end{proof}

\newpage
{\tiny 
\bibliographystyle{plain}
\bibliography{BIBLIMITED_2025}

\end{document}